\documentclass[11pt,a4paper,english]{amsart}

\usepackage[T1]{fontenc}
\usepackage[utf8]{inputenc}
\usepackage{amsmath}
\usepackage{amssymb}
\usepackage{amsthm}
\usepackage{latexsym}
\usepackage{amsfonts}

\usepackage[onehalfspacing]{setspace} 
\setdisplayskipstretch{.421} 
\newlength{\abovebis} 
\newlength{\belowbis} 
\newlength{\aboveshortbis} 
\newlength{\belowshortbis} 
\everydisplay\expandafter{%
  \the\everydisplay 
  \advance\abovedisplayskip\abovebis 
  \advance\belowdisplayskip\belowbis 
  \advance\abovedisplayshortskip\aboveshortbis 
  \advance\belowdisplayshortskip\belowshortbis 
}

\allowdisplaybreaks[1]

\def\R{\mathbb{R}}

\def\N{\mathbb{N}}
\def\C{\mathbb{C}}

\def\Ree{\mathrm{Re}}
\def\Imm{\mathrm{Im}}
\def\Div{\mathrm{div}}
\def\supp{\mathrm{supp}\,}

\theoremstyle{plain}

\newtheorem{lem}{Lemma}[section]
\newtheorem{theo}[lem]{Theorem}

\newtheorem{prob}{Problem}
\newtheorem{prop}[lem]{Proposition}

\theoremstyle{definition}

\newtheorem*{rem}{Remark}

\numberwithin{equation}{section}

\begin{document}
\title[New global stability in 2D]{New global stability estimates for the Calder\'on problem in  two dimensions}
\author{Matteo Santacesaria}
\address[M. Santacesaria]{Centre de Mathématiques Appliquées, \'Ecole Polytechnique, 91128, Palaiseau, France}
\email{santacesaria@cmap.polytechnique.fr}
\begin{abstract}
We prove a new global stability estimate for the Gel'fand-Calder\'on inverse problem on a two-dimensional bounded domain. Specifically, the inverse boundary value problem for the equation $-\Delta \psi + v\, \psi = 0$ on $D$ is analysed, where $v$ is a smooth real-valued potential of conductivity type defined on a bounded planar domain $D$. The main feature of this estimate is that it shows that the more a potential is smooth, the more its reconstruction is stable. Furthermore, the stability is proven to depend exponentially on the smoothness, in a sense to be made precise. The same techniques yield a similar estimate for the Calder\'on problem for the  electrical impedance tomography.
\end{abstract}

\maketitle

\section{Introduction}
Let $D \subset \R^2$ be a bounded domain equipped with a potential given by a function $v \in L^{\infty}(D)$.  The corresponding Dirichlet-to-Neumann map is the operator $\Phi: H^{1/2}(\partial D) \to H^{-1/2}(\partial D)$, defined by
\begin{equation}
\Phi(f) = \left. \frac{\partial u}{\partial \nu}\right|_{\partial D},
\end{equation}
where $f \in H^{1/2}(\partial D)$, $\nu$ is the outer normal of $\partial D$, and $u$ is the $H^1(D)$-solution of the Dirichlet problem
\begin{equation}
(-\Delta + v)u = 0 \; \textrm{on} \; D, \; \; \; u|_{\partial D}=f.
\label{schr}
\end{equation}
Here we have assumed that
\begin{equation} \label{direig}
0 \textrm{ is not a Dirichlet eigenvalue for the operator } - \Delta + v \textrm{ in } D.
\end{equation}

The following inverse boundary value problem arises from this construction:
\begin{prob}
Given $\Phi$, find $v$ on $D$.
\end{prob}

This problem can be considered as the Gel’fand inverse boundary value problem for the Schr\"odinger equation at zero energy (see \cite{G}, \cite{N1}) as well as a generalization of the Calder\'on problem for the electrical impedance tomography (see \cite{C}, \cite{N1}), in two dimensions.

It is convenient to recall how the above problem generalises the inverse conductivity problem proposed by Calder\'on. In the latter, $D$ is a body equipped with an isotropic conductivity $\sigma(x) \in L^{\infty}(D)$ (with $\sigma \geq \sigma_{\textrm{min}} > 0$),
\begin{align} \label{vcond}
&v(x) = \frac{\Delta \sigma^{1/2}(x)}{\sigma^{1/2}(x)}, \qquad x \in D,\\ \label{opchange}
&\Phi = \sigma^{-1/2}\left( \Lambda \sigma^{-1/2}  + \frac{\partial \sigma^{1/2}}{\partial \nu}\right),
\end{align}
where $\sigma^{-1/2}$, ${\partial \sigma^{1/2}}/{\partial \nu}$ in \eqref{opchange} denote the multiplication operators by the functions $\sigma^{-1/2}|_{\partial D}$, ${\partial \sigma^{1/2}}/{\partial \nu}|_{\partial D}$, respectively and $\Lambda$ is the voltage-to-current map on $\partial D$, defined as
\begin{equation}
\Lambda f = \left. \sigma \frac{\partial u}{\partial \nu}\right|_{\partial D},
\end{equation}
where $f \in H^{1/2}(\partial D)$, $\nu$ is the outer normal of $\partial D$, and $u$ is the $H^1(D)$-solution of the Dirichlet problem
\begin{equation} \label{conser}
\Div(\sigma \nabla u) = 0 \; \textrm{on} \; D, \; \; \; u|_{\partial D}=f.
\end{equation}
Indeed, the substitution $u = \tilde  u \sigma^{-1/2}$ in \eqref{conser} yields $(-\Delta + v)\tilde u = 0$ in $D$ with $v$ given by \eqref{vcond}.
The following problem is called the Calder\'on problem:
\begin{prob}
Given $\Lambda$, find $\sigma$ on D. 
\end{prob}
We remark that Problems 1 and 2 are not overdetermined, in the sense that we consider the reconstruction of a real-valued function of two variables from real-valued inverse problem data dependent on two variables. In addition, the history of inverse problems for the two-dimensional Schr\"odinger equation at fixed energy goes back to \cite{DKN}.\smallskip

There are several questions to be answered in these inverse problems: to prove the uniqueness of their solutions (e.g. the injectivity of the map $v \to \Phi$ for Problem 1), the reconstruction and the stability of the inverse map.

In this paper we study interior stability estimates for the two problems. Let us consider, for instance, Problem 1 with a potential of conductivity type. We want to prove that given two Dirichlet-to-Neumann operators, respectively $\Phi_1$ and $\Phi_2$, corresponding to potentials, respectively $v_1$ and $v_2$ on $D$, we have that
\begin{equation} \nonumber
\|v_1 - v_2\|_{L^{\infty}(D)} \leq \omega \left( \| \Phi_1 - \Phi_2\|_{H^{1/2} \to H^{-1/2}}\right),
\end{equation}
where the function $\omega(t) \to 0$ as fast as possible as $t \to 0$. For Problem 2 similar estimates are considered.

There is a wide literature on the Gel'fand-Calder\'on inverse problem. In the case of complex-valued potentials the global injectivity of the map $v \to \Phi$ was firstly proved in \cite{N1} for $D \subset \R^d$ with $d \geq 3$ and in \cite{B} for $d = 2$ with $v \in L^p$: in particular, these results were obtained by the use of global reconstructions developed in the same papers. A global stability estimate for Problem 1 and 2 for $d \geq 3$ was first found by Alessandrini in \cite{A}; this result was recently improved in \cite{N2}. In the two-dimensional case the first global stability estimate for Problem 1 was given in \cite{NS}.

Global results for Problem 2 in the two dimensional case have been found much earlier than for Problem 1. In particular, global uniqueness was first proved in \cite{Na} for conductivities in the $W^{2,p}(D)$ class ($p >1$) and after in \cite{AP} for $L^{\infty}$ conductivities. The first global stability result was given in \cite{L}, where a logarithmic estimate is obtained for conductivities with two continuous derivatives. This result was improved in \cite{BFR}, where the same kind of estimate is obtained for H\"older continuous conductivities.

The research line delineated above is devoted to prove stability estimates for the least regular potentials/conductivities possible. Here, instead, we focus on the opposite situation, i.e. smooth potentials/conductivities, and try to answer another question: how the stability estimates vary with respect to the smoothness of the potentials/conductivities.

The results, detailed below, also constitute a progress for the case of non-smooth potentials: they indicate stability dependence of the smooth part of a singular potential with respect to boundary value data.\smallskip

We will assume for simplicity that
\begin{equation} \label{cv1}
\begin{split}
&D \text{ is an open bounded domain in } \R^2, \qquad \partial D \in C^2, \\
&v \in W^{m,1}(\R^2) \text{ for some } m > 2, \qquad \mathrm{supp} \; v \subset D,
\end{split}
\end{equation} 
where
\begin{align}
&W^{m,1}(\R^2) = \{ v \; : \; \partial^J v \in L^1(\R^2),\; |J| \leq m \}, \qquad m \in \N \cup \{0\},\\ \nonumber
&J \in (\N \cup \{0\})^2, \qquad |J| = J_1+J_2, \qquad \partial^J v(x) = \frac{\partial^{|J|} v(x)}{\partial x_1^{J_1} \partial x_2^{J_2}}.
\end{align}
Let
\begin{equation} \nonumber
\|v\|_{m,1} = \max_{|J| \leq m} \| \partial^J v \|_{L^1(\R^2)}.
\end{equation}
The last (strong) hypothesis is that we will consider only potentials of conductivity type, i.e.
\begin{equation} \label{cv2}
v = \frac{\Delta \sigma^{1/2}}{\sigma^{1/2}}, \text{ for some } \sigma \in L^{\infty}(D), \; \text{with } \sigma \geq \sigma_{\textrm{min}} > 0.
\end{equation}
The main results are the following.

\begin{theo} \label{maintheo}
Let the conditions \eqref{direig}, \eqref{cv1}, \eqref{cv2} hold for the potentials $v_1, v_2$, where $D$ is fixed, and let $\Phi_1$ , $\Phi_2$ be the corresponding Dirichlet-to-Neumann operators. Let $\|v_j\|_{m,1} \leq N$, $j=1,2$, for some $N >0$. Then there exists a constant $C = C(D, N,m)$ such that
\begin{equation}
\| v_2 - v_1\|_{L^{\infty}(D)} \leq C (\log(3 + \| \Phi_2 - \Phi_1 \|^{-1} ))^{-\alpha},
\end{equation}
where $\alpha= m-2$ and $\| \Phi_2 - \Phi_1 \| = \| \Phi_2 - \Phi_1 \|_{H^{1/2} \to H^{-1/2}}$.
\end{theo}

\begin{theo} \label{coro}
Let $\sigma_1, \sigma_2$ be two isotropic conductivities such that $\Delta( \sigma_j^{1/2}) / \sigma_j^{1/2}$ satisfies conditions \eqref{cv1}, where $D$ is fixed and $0<\sigma_{\min} \leq \sigma_j \leq \sigma_{\max} < +\infty$ for $j=1,2$ and some constants $\sigma_{\min}$ and $\sigma_{\max}$. Let $\Lambda_1$ , $\Lambda_2$ be the corresponding Dirichlet-to-Neumann operators and $\|\Delta( \sigma_j^{1/2}) / \sigma_j^{1/2}\|_{m,1} \leq N$, $j=1,2$, for some $N >0$. We suppose, for simplicity, that $\supp ( \sigma_j -1) \subset D$ for $j=1,2$. Then, for any $\alpha < m$ there exists a constant $C = C(D, N,\sigma_{\min}, \sigma_{\max}, m,\alpha)$ such that
\begin{equation}
\| \sigma_2 - \sigma_1\|_{L^{\infty}(D)} \leq C (\log(3 + \| \Lambda_2 - \Lambda_1 \|^{-1} ))^{-\alpha},
\end{equation}
where $\| \Lambda_2 - \Lambda_1 \| = \| \Lambda_2 - \Lambda_1 \|_{H^{1/2} \to H^{-1/2}}$.
\end{theo}

The main feature of these estimates is that, as $m \to +\infty$, we have $\alpha \to +\infty$. In addition we would like to mention that, under the assumptions of Theorems \ref{maintheo} and \ref{coro}, according to instability estimates of Mandache \cite{M} and Isaev \cite{I}, our results are almost optimal. Note that, in the linear approximation near the zero potential, Theorem \ref{maintheo} (without condition \eqref{cv2}) was proved in \cite{NN}. In dimension $d \geq 3$ a global stability estimate similar to our result (with respect to dependence on smoothness) was proved in \cite{N2}.\smallskip

The proof of Theorem \ref{maintheo} relies on the $\bar \partial$-techniques introduced by Beals--Coifman \cite{BC}, Henkin--R. Novikov \cite{HN}, Grinevich--S. Novikov \cite{GN} and developed by R. Novikov \cite{N1} and Nachman \cite{Na} for solving the Calder\'on problem in two dimensions.

The Novikov--Nachman method starts with the construction of a special family of solutions $\psi(x,\lambda)$ of equation \eqref{schr}, which was originally introduced by Faddeev in \cite{F}. These solutions have an exponential behaviour depending on the complex parameter $\lambda$ and they are constructed via some function $\mu(x,\lambda)$ (see \eqref{defmu}). One of the most important property of $\mu(x,\lambda)$ is that it satisfies a $\bar \partial$-equation with respect to the variable $\lambda$ (see equation \eqref{dbar}), in which appears the so-called Faddeev generalized scattering amplitude $h(\lambda)$ (defined in \eqref{defh}). On the contrary, if one knows $h(\lambda)$ for every $\lambda \in \C$, it is possible to recover $\mu(x,\lambda)$ via this $\bar \partial$-equation. Starting from these arguments we will prove that the map $h(\lambda) \to \mu(z,\lambda)$ satisfies an H\"older condition, uniformly in the space variable $z$. This is done in Section \ref{sechv}. 

Another part of the method relates the scattering amplitude $h(\lambda)$ to the Dirichlet-to-Neumann operator $\Phi$. In the present paper this is done using the Alessandrini identity (see \cite{A}) and an estimate of $h(\lambda)$ for high values of $|\lambda|$ given in \cite{N4}. We find that the map $\Phi \to h$ has logarithmic stability in some natural norm (Proposition \ref{prophdir}). This is explained in Section \ref{secphih}.

The final part of the method for the two problems is quite different. For Problem 2, in order to recover $\sigma(x)$ from $\mu(x,\lambda)$, we use a limit found for the first time in \cite{Na}. Instead, for Problem 1, we use an explicit formula for $v(x)$ which involves the scattering amplitude $h(\lambda)$, $\mu(x,\lambda)$ and its first (complex) derivative with respect to $z = x_1+ix_2$ (see formula \eqref{expl}). The two results are presented in section \ref{secpf} and yield the proofs of Theorems \ref{maintheo} and \ref{coro}.

This work was fulfilled in the framework of researches under the direction of R. G. Novikov.
\section{Preliminaries}
In this section we recall some definitions and properties of the Faddeev functions, the above-mentioned family of solutions of equation \eqref{schr}, which will be used throughout all the paper.
\smallskip

Following \cite{Na}, we fix some $1 < p < 2$ and define $\psi(x,k)$ to be the solution (when it exists unique) of
\begin{equation} \label{equa}
(-\Delta + v) \psi(x,k) = 0 \text{ in } \R^2,
\end{equation}
with $e^{-ixk} \psi(x,k) - 1 \in W^{1,\tilde p}(\R^2)= \{ u \; : \; \partial^{J}u \in L^{\tilde p}(\R^2),\; |J| \leq 1 \}$, where $x=(x_1,x_2) \in\R^2$, $k=(k_1,k_2) \in \mathcal{V} \subset \C^2$,
\begin{align} \label{defV}
\mathcal{V} = \{k \in \C^2 \, : \, k^2=k_1^2+k_2^2 =0 \}
\end{align} and
\begin{equation} \label{deftilde}
\frac{1}{\tilde p} = \frac{1}{p} - \frac 1 2.
\end{equation}
The variety $\mathcal{V}$ can be written as $\{ (\lambda, i\lambda) : \lambda \in \C \} \cup \{ (\lambda, -i\lambda) :  \lambda \in \C \}$. We henceforth denote $\psi(x,(\lambda, i \lambda))$ by $\psi(x,\lambda)$ and observe that, since $v$ is real-valued, uniqueness for \eqref{equa} yields $\psi(x, (-\bar \lambda, i \bar \lambda)) = \overline{\psi(x,(\lambda, i \lambda))} =\overline{\psi(x,\lambda)}$ so that, for reconstruction and stability purpose, it is sufficient to work on the sheet $k= (\lambda, i \lambda)$.

We now identify $\R^2$ with $\C$ and use the coordinates $z= x_1 + i x_2, \; \bar z = x_1 - i x_2$,
\begin{equation} \nonumber
\frac{\partial}{\partial z}=\frac{1}{2}\left( \frac{\partial}{\partial x_1}-i \frac{\partial}{\partial x_2}\right), \quad \frac{\partial}{\partial \bar z}=\frac{1}{2}\left( \frac{\partial}{\partial x_1}+i \frac{\partial}{\partial x_2}\right),
\end{equation}
where $(x_1, x_2) \in \R^2$.

Then we define
\begin{align}\label{defpsilambda}
\psi(z,\lambda) &= \psi(x,\lambda), \\ \label{defmu}
\mu(z,\lambda) &= e^{-iz\lambda} \psi(z,\lambda), \\ \label{defh}
h(\lambda) &= \int_D e^{i \bar z \bar \lambda} v(z) \psi(z,\lambda) d\Ree z \, d \Imm z,
\end{align}
for $z, \lambda \in \C$.

Throughout all the paper $c(\alpha, \beta, \ldots)$ is a positive constant depending on parameters $\alpha, \beta, \ldots$ \smallskip

We now restate some fundamental results about Faddeev functions. In the following statement $\psi_0$ denotes $\sigma^{1/2}$.

\begin{prop}[see \cite{Na}] \label{mainprop}
Let $D \subset \R^2$ be an open bounded domain with $C^2$ boundary, $v \in L^p(\R^2)$, $1 < p <2$, $\mathrm{supp} \, v \subset D$, $\|v\|_{L^p(\R^2)} \leq N$, be such that there exists a real-valued $\psi_0 \in L^{\infty}(\R^2)$ with $v = (\Delta \psi_0) / \psi_0, \; \psi_0(x) \geq c_0 > 0$ and $\psi_0 \equiv 1$ outside $D$. Then, for any $\lambda \in \C$ there is a unique solution $\psi(z,\lambda)$ of \eqref{equa} with $e^{-iz \lambda}\psi(\cdot, \lambda) - 1$ in $L^{\tilde p} \cap L^{\infty}$ ($\tilde p$ is defined in \eqref{deftilde}). Furthermore, $e^{-iz \lambda}\psi(\cdot, \lambda) - 1 \in W^{1,\tilde p}(\R^2)$ and
\begin{equation} \label{estpsi}
\|e^{-iz \lambda}\psi(\cdot, \lambda) - 1\|_{W^{s,\tilde p}} \leq c(p,s) N |\lambda|^{s-1},
\end{equation}
for $0 \leq s \leq 1$ and $\lambda$ sufficiently large.\smallskip

The function $\mu(z,\lambda)$ defined in \eqref{defmu} satisfies the equation
\begin{equation}\label{dbar}
\frac{\partial \mu(z,\lambda)}{\partial \bar \lambda} = \frac{1}{4 \pi \bar \lambda}h(\lambda)e_{-\lambda}(z) \overline{\mu(z,\lambda)}, \qquad z,\lambda \in \C,
\end{equation}
in the $W^{1,\tilde p}$ topology, where $h(\lambda)$ is defined in \eqref{defh} and the function $e_{-\lambda} (z)$ is defined as follows:
\begin{equation}\label{defexp}
e_{\lambda}(z) = e^{i(z\lambda + \bar z \bar \lambda)}.\smallskip
\end{equation}
In addition, the functions $h(\lambda)$ and $\mu(z,\lambda)$ satisfy 
\begin{gather} \label{prophlp}
\left\|\frac{h(\lambda)}{\bar \lambda}\right\|_{L^r(\R^2)}\! \! \! \! \! \! \! \! \leq c(r,  N), \; \text{for all } r \in (\tilde p', \tilde p), \quad \frac{1}{ \tilde p} + \frac{1}{\tilde p'} = 1,\\ \label{normmu}
\qquad \sup_{z \in \C} \|\mu(z,\cdot) - 1\|_{L^r(\C)}\leq c(r,D,  N), \qquad \text{for all } r \in (p', \infty]
\end{gather}
and
\begin{gather} \label{esthnear}
|h(\lambda)| \leq c(p,D, N) |\lambda|^{\varepsilon},\\ \label{estmunear}
\| \mu(\cdot, \lambda) - \psi_0\|_{W^{1,\tilde p}} \leq c(p,D,  N) |\lambda|^{\varepsilon},
\end{gather}
for $\lambda \leq \lambda_0(p,D,N)$ and $0 < \varepsilon < \frac{2}{p'}$, where $\frac{1}{p} + \frac{1}{p'} = 1$.
\end{prop}

\begin{rem}
Equation \eqref{dbar} means that $\mu$ is a generalised analytic function in $\lambda \in \C$ (see \cite{V}). In two-dimensional inverse scattering for the Schr\"odinger equation, the theory of generalised analytic functions was used for the first time in \cite{GN}.
\end{rem}

We recall that if $v \in W^{m,1}(\R^2)$ with $\supp v \subset D$, then $\|\hat{v} \|_m < +\infty$, where
\begin{gather}
\hat v(p) = (2 \pi)^{-2} \int_{\R^2} e^{ipx} v(x) dx, \qquad p \in \C^2, \\
\|u\|_m = \sup_{p \in \R^2} | (1+|p|^2)^{m/2} u(p)|,
\end{gather}
for a test function $u$.

In addition, if $v \in W^{m,1}(\R^2)$ with $\supp v \subset D$ and $m > 2$, we have, by Sobolev embedding, that
\begin{equation} \label{normvvv}
\|v\|_{L^{\infty}(D)} \leq c(D) \|v\|_{m,1},
\end{equation}
so, in particular, the hypothesis $v \in L^p(\R^2)$, $\supp v \subset D$, in the statement of Proposition \ref{mainprop} is satisfied for every $1<p<2$ (since $D$ is bounded).

The following lemma is a variation of a result in \cite{N4}:
\begin{lem} \label{lemh}
Under the assumption \eqref{cv1}, there exists $R=R(m,\|\hat v\|_m) > 0$ such that
\begin{equation}
|h(\lambda)| \leq 8 \pi^2 \|\hat v\|_m ( 1+4|\lambda|^2)^{-m/2}, \qquad \text{for } |\lambda| > R.
\end{equation}
\end{lem}
\begin{proof}
We consider the function $H(k,p)$ defined as
\begin{equation}
H(k,p) = \frac{1}{(2 \pi)^2}\int_{\R^2}e^{i(p-k)x}v(x)\psi(x,k) dx,
\end{equation}
for $k \in \mathcal{V}$ (where $\mathcal{V}$ is defined in \eqref{defV}), $p \in \R^2$ and $\psi(x,k)$ as defined at the beginning of this section.

We deduce that $h(\lambda) = (2\pi)^2 H(k(\lambda), k(\lambda) + \overline{k(\lambda)})$, for $k(\lambda) = (\lambda, i \lambda)$. By \cite[Corollary 1.1]{N4} we have
\begin{equation}
|H(k,p)| \leq 2 \|\hat v\|_m ( 1+p^2)^{-m/2} \qquad \text{for } |\lambda| > R,
\end{equation}
for $R = R(m,\|\hat v\|_m) >0$ and then the proof follows.
\end{proof}

We restate \cite[Lemma 2.6]{BBR}, which will be useful in section \ref{sechv}.
\begin{lem}[\cite{BBR}] \label{lemtech}
Let $a \in L^{s_1}(\R^2) \cap L^{s_2}(\R^2)$, $1 < s_1 <2 < s_2 < \infty$ and $b \in L^s( \R^2)$, $1 < s <2$. Assume $u$ is a function in $L^{\tilde s}(\R^2)$, with $\tilde s$ defined as in \eqref{deftilde}, which satisfies
\begin{equation}
\frac{\partial u (\lambda)}{\partial \bar \lambda} = a(\lambda) \bar u(\lambda) + b(\lambda), \qquad \lambda \in \C.
\end{equation}
Then there exists $c >0$ such that
\begin{equation}
\|u\|_{L^{\tilde s}} \leq c \|b\|_{L^s} \exp(c (\|a\|_{L^{s_1}}+\|a\|_{L^{s_2}})).
\end{equation}
\end{lem}

We will make also use of the well-known H\"older's inequality, which we recall in a special case: for $f \in L^p(\C)$, $g \in L^q(\C)$ such that $1 \leq p,q\leq \infty$, $1\leq r < \infty$, $1/p+1/q = 1/r$, we have
$$ \|fg\|_{L^r(\C)}\leq \|f\|_{L^p(\C)}\|g\|_{L^q(\C)}.$$

\section{From $\Phi$ to $h(\lambda)$} \label{secphih}

\begin{lem} \label{lemesth}
Let the condition \eqref{cv1} holds. Then we have, for $p \geq 1$,
\begin{align} \label{esthlg}
\left\| \frac{h(\lambda)}{\bar \lambda} \right\|_{L^p(|\lambda| > R)}\! \! \! \! \! \! \!  &\leq c(p,m)\|\hat v\|_m \frac{1}{R^{m+1-2/p}},\\ \label{esthg}
\|h\|_{L^p(|\lambda| > R)} &\leq c(p,m)\|\hat v\|_m \frac{1}{R^{m-2/p}},
\end{align} 
where $R$ is as in Lemma \ref{lemh}.
\end{lem}
\begin{proof}
It's a corollary of Lemma \ref{lemh}. Indeed we have
\begin{align}
&\left\| \frac{h(\lambda)}{\bar \lambda} \right\|_{L^p(|\lambda| > R)}^p \leq c \|\hat v\|_m^p \int_{r > R} r^{1-mp-p} dr = \frac{c(p,m)\|\hat v\|_m^p }{R^{(m+1)p-2}},
\end{align}
which gives \eqref{esthlg}. The proof of \eqref{esthg} is analogous.
\end{proof}

\begin{lem} \label{lemdifh}
Let $D \subset \{ x \in \R^2 \, : \, |x| \leq l\}$, $v_1, v_2$ be two potentials satisfying \eqref{direig}, \eqref{cv1}, \eqref{cv2}, let $\Phi_1, \Phi_2$ the corresponding Dirichlet-to-Neumann operator and $h_1, h_2$ the corresponding generalised scattering amplitude. Let $\|v_j\|_{m,1} \leq N$, $j=1,2$. Then we have
\begin{equation} \label{estdifh}
|h_2(\lambda) - h_1(\lambda)|\leq c(D,N)e^{2l|\lambda|}\|\Phi_2 - \Phi_1\|_{H^{1/2}\to H^{-1/2}}, \quad \lambda \in \C.
\end{equation}
\end{lem}
\begin{proof}
We have the following identity:
\begin{equation} \label{aless}
h_2(\lambda) - h_1(\lambda) = \int_{\partial D}\overline{\psi_1(z,\lambda)}(\Phi_2 - \Phi_1)\psi_2(z,\lambda) |dz|,
\end{equation}
where $\psi_j(z,\lambda)$ are the Faddeev functions associated to the potential $v_j$, $j=1,2$. This identity is a particular case of the one in \cite[Theorem 1]{N3}: we refer to that paper for a proof.

From this identity we have:
\begin{align} \label{estlem1}
|h_2(\lambda) - h_1(\lambda)| \leq \|\psi_1(\cdot,\lambda)\|_{H^{1/2}(\partial D)}\|\Phi_2 - \Phi_1\|_{H^{1/2}\to H^{-1/2}} \|\psi_2(\cdot,\lambda)\|_{H^{1/2}(\partial D)}.
\end{align}
Now take $\tilde p>2$ and use the trace theorem to get
\begin{align*}
&\|\psi_j(\cdot,\lambda)\|_{H^{1/2}(\partial D)} \leq C \|\psi_j(\cdot,\lambda)\|_{W^{1,\tilde p}(D)} \leq C e^{l|\lambda|}\|e^{-iz\lambda}\psi_j(\cdot,\lambda)\|_{W^{1,\tilde p}(D)}\\
&\qquad \leq  C e^{l|\lambda|}\left(\|e^{-iz\lambda}\psi_j(\cdot,\lambda)-1\|_{W^{1,\tilde p}(D)}+\|1\|_{W^{1,\tilde p}(D)}\right), \qquad j=1,2,
\end{align*}
which from \eqref{estpsi} and \eqref{normmu} is bounded by $C(D,N) e^{l|\lambda|}$. These estimates together with \eqref{estlem1} give \eqref{estdifh}.
\end{proof}
The main results of this section are the following propositions:
\begin{prop} \label{prophdir}
Let $v_1, v_2$ be two potentials satisfying \eqref{direig}, \eqref{cv1}, \eqref{cv2}, let $\Phi_1, \Phi_2$ the corresponding Dirichlet-to-Neumann operator and $h_1, h_2$ the corresponding generalised scattering amplitude. Let $0< \varepsilon < 1$, $1 < p < \frac{2}{1-\varepsilon}$ and $\|v_j\|_{m,1} \leq N$, $j=1,2$. Then there exists a constant $c= c(D, N,m,p)$ such that
\begin{equation} \label{mainesth}
\left\| \frac{h_2(\lambda) - h_1(\lambda)}{\bar \lambda}\right\|_{L^p(\C)} \leq c \log(3 + \|\Phi_2 - \Phi_1\|_{H^{1/2}\to H^{-1/2}}^{-1})^{-(m+1-2/p)}.
\end{equation}
\end{prop}

\begin{prop} \label{prophdir2}
Let $v_1, v_2, \Phi_1, \Phi_2, h_1, h_2$ be as in Proposition \ref{prophdir}. Let $p \geq 1$ and $\|v_j\|_{m,1} \leq N$, $j=1,2$. Then there exists a constant $c= c(D, N,m,p)$ such that
\begin{equation} \label{mainesth2}
\|h_2 - h_1\|_{L^p(\C)} \leq c \log(3 + \|\Phi_2 - \Phi_1\|_{H^{1/2}\to H^{-1/2}}^{-1})^{-(m-2/p)}.
\end{equation}
\end{prop}

\begin{proof}[Proof of Proposition \ref{prophdir}]
Let choose $a, b >0$, $a$ close to $0$ and $b$ big to be determined and let 
\begin{equation} \label{defdelta}
\delta = \|\Phi_2 - \Phi_1\|_{H^{1/2}\to H^{-1/2}}.
\end{equation}
We split down the left term of \eqref{mainesth} as follows:
\begin{align*}
\left\| \frac{h_2(\lambda) - h_1(\lambda)}{\bar \lambda}\right\|_{L^p(\C)}\! \! \! \! \! &\leq \left\| \frac{h_2(\lambda) - h_1(\lambda)}{\bar \lambda}\right\|_{L^p(|\lambda| < a)} \! \! \! + \left\| \frac{h_2(\lambda) - h_1(\lambda)}{\bar \lambda}\right\|_{L^p(a < |\lambda|< b)}\\
&\qquad + \left\| \frac{h_2(\lambda) - h_1(\lambda)}{\bar \lambda}\right\|_{L^p(|\lambda| > b)}.
\end{align*}
From \eqref{esthnear} we obtain
\begin{align} \label{est1}
\left\| \frac{h_2(\lambda) - h_1(\lambda)}{\bar \lambda}\right\|_{L^p(|\lambda| < a)} \! \! \! \! \!\! \! \! \! \! &\leq c(D,N,p)\left(\int_{|\lambda| < a} |\lambda|^{(\varepsilon-1)p}d\Ree \lambda\, d \Imm \lambda \right)^{\frac 1 p}\\ \nonumber
&=c(D,N,p) a^{\varepsilon-1+2/p}.
\end{align}
From Lemma \ref{lemdifh} and \eqref{defdelta} we get
\begin{equation} \label{est2}
\left\| \frac{h_2(\lambda) - h_1(\lambda)}{\bar \lambda}\right\|_{L^p(a < |\lambda|< b)} \leq c(D,N) \left(\frac{\delta}{a^{1-2/p}}+\delta e^{2lb}\right).
\end{equation}
From Lemma \ref{lemesth}
\begin{equation} \label{est3}
\left\| \frac{h_2(\lambda) - h_1(\lambda)}{\bar \lambda}\right\|_{L^p(|\lambda| > b)} \leq \frac{c(N)}{b^{m+1-2/p}}.
\end{equation}
We now define
\begin{equation}
a =\log(3+\delta^{-1})^{-\frac{m+1-2/p}{\varepsilon-1+2/p}}, \qquad b=\beta \log(3+\delta^{-1}),
\end{equation}
for $0 < \beta < 1/(2l)$, in order to have \eqref{est1} and \eqref{est3} of the order $\log(3+\delta^{-1})^{-(m+1-2/p)}$. We also choose $\bar \delta < 1$ such that for every $\delta \leq \bar \delta$, $a$ is sufficiently small in order to have \eqref{esthnear} (which yields \eqref{est1}), $b \geq R$ (with $R$ as in Lemma \ref{lemh}) and also
\begin{equation}
\frac{\delta}{a^{1-2/p}}= \delta \log(3 + \delta^{-1})^{\left(\frac{m+1-2/p}{\varepsilon - 1 + 2/p}\right) (1-2/p)} < \log(3 + \delta^{-1})^{-(m+1-2/p)}.
\end{equation}
Thus we obtain
\begin{align}
&\left\| \frac{h_2(\lambda) - h_1(\lambda)}{\bar \lambda}\right\|_{L^p(\C)} \leq \frac{c(D,N,p)}{\log(3+\delta^{-1})^{m+1-2/p}}\\ \nonumber
&\qquad + c(D,N) \delta (3+\delta^{-1})^{2l\beta},
\end{align}
for $\delta \leq \bar \delta$, $0 < \beta < 1/(2l)$. As $\delta (3+\delta^{-1})^{2l\beta} \to 0$ for $\delta \to 0$ more rapidly than the other term, we obtain that
\begin{equation}\label{esthdir}
\left\| \frac{h_2(\lambda) - h_1(\lambda)}{\bar \lambda}\right\|_{L^p(\C)} \leq \frac{c(D,N,m,p,\beta)}{\log(3+\delta^{-1})^{m+1-2/p}},
\end{equation}
for $\delta \leq \bar \delta$, $0 < \beta < 1/(2l)$.

Estimate \eqref{esthdir} for general $\delta$ (with modified constant) follows from \eqref{esthdir} for $\delta \leq \bar \delta$ and the property \eqref{prophlp} of the scattering amplitude. This completes the proof of Proposition \ref{prophdir}.
\end{proof}
\begin{proof}[Proof of Proposition \ref{prophdir2}]
We follow almost the same scheme as in the proof of Proposition \ref{prophdir}.
Let choose $b >0$ big to be determined and let 
\begin{equation} \label{defdelta2}
\delta = \|\Phi_2 - \Phi_1\|_{H^{1/2}\to H^{-1/2}}.
\end{equation}
We split down the left term of \eqref{mainesth2} as follows:
\begin{align*}
\|h_2 - h_1\|_{L^p(\C)} \leq \|h_2 - h_1\|_{L^p(|\lambda|<b)} + \|h_2 - h_1\|_{L^p(|\lambda| \geq b)}.
\end{align*}
From Lemma \ref{lemdifh} we obtain
\begin{equation}
\|h_2 - h_1\|_{L^p(|\lambda|<b)} \leq c(D,N,p)\delta b^{1/p}e^{2l b},
\end{equation}
and from \eqref{esthg}
\begin{equation}
\|h_2 - h_1\|_{L^p(|\lambda| \geq b)} \leq c(N,p,m) \frac{1}{b^{m-2/p}}.
\end{equation}
Define $b = \beta \log(3 + \delta^{-1})$ for $0< \beta< 1/(2l)$. Let $\bar \delta < 1$ such that for $\delta \leq \bar \delta$ we have that $b > R$, where $R$ is defined in Lemma \ref{lemh}.

Then we have, for $\delta \leq \bar \delta$,
\begin{align*}
&\|h_2 - h_1\|_{L^p(\C)} \leq c(D,N,m,p)\delta (1+\delta^{-1})^{2l \beta} (\beta \log(3+\delta^{-1}))^{1/p}\\
&\qquad + c(N,m,p)(\log(3+\delta^{-1}))^{-(m-2/p)}.
\end{align*}
Since $2l\beta < 1$, we have that
\begin{align*}
\delta (1+\delta^{-1})^{2l \beta} (\beta \log(3+\delta^{-1}))^{1/p} \to 0 \quad \text{for } \delta \to 0
\end{align*}
more rapidly than the other term. Thus
\begin{equation} \label{bla}
\|h_2 - h_1\|_{L^p(\C)} \leq c(D,N,m,p,\beta)(\log(3+\delta^{-1}))^{-(m-2/p)},
\end{equation}
for $\delta \leq \bar \delta$, $0<\beta < 1/(2l)$.

Estimate \eqref{bla} for general $\delta$ (with modified constant) follows from \eqref{bla} for $\delta \leq \bar \delta$ and the $L^p$-boundedness of the scattering amplitude (this because it is continuous and decays at infinity like in Lemma \ref{lemesth}). This completes the proof of Proposition \ref{prophdir2}.
\end{proof}

\section{Estimates on the Faddeev functions} \label{sechv}
\begin{lem} \label{lemestmu}
Let $v_1, v_2$ be two potentials satisfying \eqref{direig}, \eqref{cv1}, \eqref{cv2}, with $\|v_j\|_{m,1} \leq N$, $h_1, h_2$ the corresponding scattering amplitude and $\mu_1(z,\lambda), \mu_2(z,\lambda)$ the corresponding Faddeev functions. Let $1 < s < 2$, and $\tilde s$ be as in \eqref{deftilde}. Then
\begin{align} \label{estmula1}
\sup_{z \in \C} \|\mu_2(z,\cdot) - \mu_1(z,\cdot) \|_{L^{\tilde s}(\C)} &\leq c(D,N,s) \left\| \frac{h_2(\lambda) - h_1(\lambda)}{\bar \lambda}\right\|_{L^s(\C)},\\ \label{estmula2}
\sup_{z \in \C} \left\|\frac{\partial \mu_2(z,\cdot)}{\partial z} - \frac{\partial \mu_1(z,\cdot)}{\partial z} \right\|_{L^{\tilde s}(\C)}\! \! \! \! \! \! &\leq c(D,N,s)\Bigg[ \left\| \frac{h_2(\lambda) - h_1(\lambda)}{\bar \lambda}\right\|_{L^s(\C)} \\ \nonumber
&\quad  +\|h_2 - h_1\|_{L^s(\C)} \Bigg]
\end{align}
\end{lem}
\begin{proof}
We begin with the proof of \eqref{estmula1}. Let 
\begin{align} \label{defnu}
\nu(z,\lambda) &= \mu_2(z,\lambda) - \mu_1(z,\lambda).
\end{align}

From the $\bar \partial$-equation \eqref{dbar} we deduce that $\nu$ satisfies the following non-homogeneous $\bar \partial$-equation:
\begin{align} \label{dbar1}
\frac{\partial}{\partial \bar \lambda}\nu(z,\lambda) = \frac{e_{-\lambda}(z)}{4\pi}\left(\frac{h_1(\lambda)}{\bar \lambda}\overline{\nu(z,\lambda)}+ \frac{h_2(\lambda)-h_1(\lambda)}{\bar \lambda}\overline{\mu_2(z,\lambda)}\right),
\end{align}
for $\lambda \in \C$, where $e_{-\lambda}(z)$ is defined in \eqref{defexp}. Note that since, by Sobolev embedding, $v \in L^{\infty}(D)\subset L^s(D)$, we have that $\nu(z,\cdot) \in L^{\tilde s}(\C)$ for every $\tilde s > 2$ (see \eqref{normmu}). In addition, from Proposition \ref{mainprop} (see \eqref{prophlp}) we have that $h(\lambda) / \bar \lambda \in L^p(\C)$, for $ 1 < p < \infty$. Then it is possible to use Lemma \ref{lemtech} in order to obtain
\begin{align*}
\|\nu(z,\cdot)\|_{L^{\tilde s}} &\leq c(D,N,s)\left\|\overline{\mu_2(z,\lambda)} \frac{h_2(\lambda) - h_1(\lambda)}{\bar \lambda}\right\|_{L^s(\C)}\\
&\leq c(D,N,s) \sup_{z \in \C}\|\mu_2(z,\cdot)\|_{L^{\infty}} \left\| \frac{h_2(\lambda) - h_1(\lambda)}{\bar \lambda}\right\|_{L^s(\C)}\\
&\leq c(D,N,s) \left\| \frac{h_2(\lambda) - h_1(\lambda)}{\bar \lambda}\right\|_{L^s(\C)},
\end{align*}
where we used again the property \eqref{normmu} of $\mu_2(z,\lambda)$.

Now we pass to \eqref{estmula2}. To simplify notations we write, for $z,\lambda \in \C$,
\begin{equation} \nonumber
\mu_z^j(z,\lambda) = \frac{\partial \mu_j(z,\lambda)}{\partial z}, \quad \mu_{\bar z}^j(z,\lambda) = \frac{\partial \mu_j(z,\lambda)}{\partial \bar z},\quad j=1,2.
\end{equation}

From the $\bar \partial$-equation \eqref{dbar} we have that $\mu_z^j$ and $\mu_{\bar z}^j$ satisfy the following system of non-homogeneous $\bar \partial$-equations, for $j=1,2$:
\begin{align*}
\frac{\partial}{\partial \bar \lambda}\mu_z^j(z,\lambda) = \frac{e_{-\lambda}(z)}{4\pi}\frac{h_j(\lambda)}{\bar \lambda}\left(\overline{\mu_{\bar z}^j(z,\lambda)}-i\lambda \overline{\mu_j(z,\lambda)}\right),\\
\frac{\partial}{\partial \bar \lambda}\mu_{\bar z}^j(z,\lambda) = \frac{e_{-\lambda}(z)}{4\pi}\frac{h_j(\lambda)}{\bar \lambda}\left(\overline{\mu_{z}^j(z,\lambda)}-i\bar \lambda \overline{\mu_j(z,\lambda)}\right).
\end{align*}
Define now $\mu_{\pm}^j(z,\lambda) = \mu_z^j(z,\lambda) \pm \mu_{\bar z}^j(z,\lambda)$, for $j=1,2$. 
Then they satisfy the following two non-homogeneous $\bar \partial$-equations:
\begin{equation} \nonumber
\frac{\partial}{\partial \bar \lambda}\mu_{\pm}^j(z,\lambda)= \pm \frac{e_{-\lambda}(z)}{4\pi}\frac{h_j(\lambda)}{\bar \lambda}\left( \overline{\mu_{\pm}^j(z,\lambda)}\mp i(\lambda \pm \bar \lambda) \overline{\mu_j(z,\lambda)}\right).
\end{equation}
Finally define $\tau_{\pm}(z,\lambda) = \mu_{\pm}^2(z,\lambda)- \mu_{\pm}^1(z,\lambda)$. They satisfy the two non-homogeneous $\bar \partial$-equations below:
\begin{align*}
\frac{\partial}{\partial \bar \lambda}\tau_{\pm}(z,\lambda) &= \pm \frac{e_{-\lambda}(z)}{4\pi}\bigg[\frac{h_1(\lambda)}{\bar \lambda}\overline{\tau_{\pm}(z,\lambda)}+ \frac{h_2(\lambda)-h_1(\lambda)}{\bar \lambda}\overline{\mu_{\pm}^2(z,\lambda)}\\ \nonumber
&\quad \mp i \frac{\lambda \pm \bar \lambda}{\bar \lambda}\left( \left(h_2(\lambda)-h_1(\lambda)\right)\overline{\mu_2(z,\lambda)}+h_1(\lambda)\overline{\nu(z,\lambda)}\right) \bigg],
\end{align*}
where $\nu(z,\lambda)$ was defined in \eqref{defnu}.

Now remark that by \cite[Lemma 2.1]{N4} and regularity assumptions on the potentials we have that $\mu_z^j(z,\cdot), \mu_{\bar z}^j(z,\cdot) \in L^{\tilde s}(\C) \cap L^{\infty}(\C)$ for any $\tilde s > 2$, $j=1,2$. This, in particular, yields $\tau_{\pm}(z,\cdot) \in L^{\tilde s}(\C)$. These arguments, along with the above remarks on the $L^p$ boundedness of $h_j(\lambda) / \bar \lambda$, make possible to use Lemma \ref{lemtech}, which gives
\begin{align*}
\|\tau_{\pm}(z,\cdot)\|_{L^{\tilde s}(\C)}&\leq c(D,N,s)\Bigg[\left\|\frac{h_2(\lambda)-h_1(\lambda)}{\bar \lambda} \overline{\mu_{\pm}^2(z,\cdot)}\right\|_{L^s(\C)}\\
&\quad +\|(h_2(\cdot)-h_1(\cdot))\overline{\mu_2(z,\cdot)}\|_{L^s(\C)}+\|h_1(\cdot) \overline{\nu(z,\cdot)}\|_{L^s(\C)}\Bigg]\\
&\leq c(D,N,s)\Bigg[\left\|\frac{h_2(\lambda)-h_1(\lambda)}{\bar \lambda}\right\|_{L^s(\C)}+\|h_2-h_1\|_{L^s(\C)}\\
&\quad + \|h_1\|_{L^2(\C)}\|\nu(z,\cdot)\|_{L^{\tilde s}(\C)}\Bigg]\\
&\leq c(D,N,s)\Bigg[ \left\|\frac{h_2(\lambda)-h_1(\lambda)}{\bar \lambda}\right\|_{L^s(\C)}+\|h_2-h_1\|_{L^s(\C)}\Bigg],
\end{align*}
where we used H\"older's inequality (since $1/s = 1/2 + 1/ \tilde s$) and estimate \eqref{estmula1}.
The proof of \eqref{estmula2} now follows from this last inequality and the fact that $\mu^2_z-\mu^1_z = \frac{1}{2}(\tau_+-\tau_-)$.
\end{proof}
\begin{rem}
We also have proved that
\begin{align*}
\sup_{z \in \C} \left\|\frac{\partial \mu_2(z,\cdot)}{\partial \bar z} - \frac{\partial \mu_1(z,\cdot)}{\partial \bar z} \right\|_{L^{\tilde s}(\C)}\! \! \! \! \! \! &\leq c(D,N,s)\Bigg[ \left\| \frac{h_2(\lambda) - h_1(\lambda)}{\bar \lambda}\right\|_{L^s(\C)} \\ \nonumber
&\quad  +\|h_2 - h_1\|_{L^s(\C)} \Bigg].
\end{align*}
\end{rem}

We will need the following consequence of Lemma \ref{lemestmu}.
\begin{lem} \label{prophv}
Let $v_1, v_2$ be two potentials satisfying \eqref{direig}, \eqref{cv1}, \eqref{cv2}, with $\|v_j\|_{m,1} \leq N$. Let $h_1, h_2$ be the corresponding scattering amplitude and $\mu_1(z,\lambda), \mu_2(z,\lambda)$ the corresponding Faddeev functions. Let $p,p'$ such that $1<p<2<p'<\infty$, $1/p+1/p' =1$. Then
\begin{equation} \label{est23}
\|\mu_2(\cdot,0)-\mu_1(\cdot,0)\|_{L^{\infty}(D)} \leq c(D,N,p) \left\|\frac{h_2(\lambda) - h_1(\lambda)}{\bar \lambda}\right\|_{L^p(\C)\cap L^{p'}(\C)}.
\end{equation}
\end{lem}
\begin{proof}
We recall again that if $v \in W^{m,1}(\R^2)$, $m >2$, with $\supp v \subset D$ then $v \in L^p(D)$ for $p \in [1,\infty]$; in particular, from Proposition \ref{mainprop}, this yields $h(\lambda) / \bar \lambda \in L^p(\C)$, for $ 1 < p < \infty$.

We write, as in the preceding proof,
\begin{align} \label{est330}
\nu(z,\lambda) &= \mu_2(z,\lambda) - \mu_1(z,\lambda),
\end{align}
which satisfy the non-homogeneous $\bar \partial$-equations \eqref{dbar1}. From this equation we obtain
\begin{align} \label{est331}
|\nu(z,0)| &= \frac{1}{\pi}\left| \int_{\C}\frac{e_{-\lambda}(z)}{4 \pi \lambda}\frac{h_1(\lambda)}{\bar \lambda}\overline{\nu(z,\lambda)}d \Ree \lambda \, d \Imm \lambda \right.\\ \nonumber
&\quad \left. + \int_{\C}\frac{e_{-\lambda}(z)}{4 \pi \lambda}\frac{h_2(\lambda)-h_1(\lambda)}{\bar \lambda}\overline{\mu_2(z,\lambda)}d \Ree \lambda \, d \Imm \lambda \right|\\\nonumber
&\leq \frac{1}{4 \pi^2}\sup_{z \in \C}\|\nu(z,\cdot)\|_{L^r} \left\|\frac{h_1(\lambda)}{\lambda \bar \lambda}\right\|_{L^{r'}}\\ \nonumber
&\quad +\frac{1}{4 \pi^2}\sup_{z \in \C}\|\mu_2(z,\cdot)\|_{L^{\infty}}\left\|\frac{h_2(\lambda) -h_1(\lambda)}{\lambda \bar \lambda}\right\|_{L^{1}}
\end{align}
where $1/r + 1/r' =1$, $1<r' <2<r <\infty$. The number $s = 2 r/ ( r+2)$ can be chosen $s < 2$ and as close to $2$ as wanted, by taking $r$ big enough.

Then
\begin{align}
\left\|\frac{h_1(\lambda)}{\lambda \bar \lambda}\right\|_{L^{r'}(|\lambda| < R)}\leq \left\|\frac{h_1(\lambda)}{\bar \lambda}\right\|_{L^{p}} \left\|\frac{1}{\lambda }\right\|_{L^{q}(|\lambda| < R)} \leq c(N,r),
\end{align}
where we have chosen $p >2$ such that $\left\| h_1(\lambda)/{\bar \lambda}\right\|_{L^{p}} \leq c(N,p)$ from \eqref{prophlp} and also, since $1/q = 1/r'-1/p = 1-1/r-1/p$, $q$ can be chosen less than $2$ by taking $r$ big enough depending on $p$. With the same choice of $p,q$ we also obtain
\begin{align}
\left\|\frac{h_1(\lambda)}{\lambda \bar \lambda}\right\|_{L^{r'}(|\lambda| > R)}\leq \left\|\frac{h_1(\lambda)}{\bar \lambda}\right\|_{L^{q}} \left\|\frac{1}{\lambda }\right\|_{L^{p}(|\lambda| > R)} \leq c(N,r).
\end{align}
From Lemma \ref{lemestmu} with $r = \tilde s$ we get
\begin{align}
\sup_{z \in \C}\|\nu(z,\cdot)\|_{L^r} &\leq c(D,N,r) \left\| \frac{h_2(\lambda) - h_1(\lambda)}{\bar \lambda}\right\|_{L^s(\C)},
\end{align}
and from \eqref{normmu}
\begin{equation}
\sup_{z,\lambda \in \C} |\mu_2(z,\lambda)| \leq c(D,N).
\end{equation}
Finally
\begin{align} \label{est333}
\left\|\frac{h_2(\lambda) -h_1(\lambda)}{\lambda \bar \lambda}\right\|_{L^{1}} &\leq \left\|\frac{1}{\lambda}\right\|_{L^{p}(|\lambda| > R)}\left\|\frac{h_2(\lambda) -h_1(\lambda)}{\bar \lambda}\right\|_{L^{p'}}\\ \nonumber
&\quad +\left\|\frac{1}{\lambda}\right\|_{L^{p'}(|\lambda| < R)}\left\|\frac{h_2(\lambda) -h_1(\lambda)}{\bar \lambda}\right\|_{L^{p}},
\end{align}
by taking $p' = s$ and $p$ such that $1/p + 1/p' =1$.
Now \eqref{est23} follow from \eqref{est330}--\eqref{est333}; this finishes the proof of Lemma \ref{prophv}.
\end{proof}

\section{Proof of Theorems \ref{maintheo} and \ref{coro}} \label{secpf}
\begin{proof}[Proof of Theorem \ref{maintheo}]
We begin with a remark, which take inspiration from Problem 1 at non-zero energy (see, for instance, \cite{N5}).

Let $v(z)$ be a potential which satisfies the hypothesis of Theorem \ref{maintheo} and $\mu(z,\lambda)$ the corresponding Faddeev functions. Since $\mu(z,\lambda)$ satisfies \eqref{normmu}, the $\bar \partial$-equation \eqref{dbar} and $h(\lambda)$ decreases at infinity like in Lemma \ref{lemh}, it is possible to write the following development:
\begin{equation} \label{devel}
\mu(z,\lambda) = 1 + \frac{\mu_{-1}(z)}{\lambda}+O\left(\frac{1}{|\lambda|^2}\right), \qquad \lambda \to \infty,
\end{equation}
for some function $\mu_{-1}(z)$. If we insert \eqref{devel} into equation \eqref{equa}, for $\psi(z,\lambda) = e^{iz\lambda} \mu(z,\lambda)$, we obtain, letting $\lambda \to \infty$,
\begin{equation}
v(z) = 4 i \frac{\partial \mu_{-1}(z)}{\partial \bar z}, \qquad z \in \C.
\end{equation}
We can write this in a more explicit form, using the following integral equation (a consequence of \eqref{dbar}):
\begin{equation} \nonumber
\mu(z,\lambda)-1=\frac{1}{8\pi^2 i}\int_{\C}\frac{h(\lambda')}{(\lambda'-\lambda)\bar \lambda'}e_{-\lambda'}(z)\overline{\mu(z,\lambda')}d\lambda' \, d \bar \lambda'.
\end{equation}
By Lebesgue's dominated convergence (using \eqref{esthnear}) we obtain
\begin{equation}\nonumber
\mu_{-1}(z)=-\frac{1}{8\pi^2 i}\int_{\C}\frac{h(\lambda)}{\bar \lambda}e_{-\lambda}(z)\overline{\mu(z,\lambda)}d\lambda \, d \bar \lambda,
\end{equation}
and the explicit formula
\begin{equation} \label{expl}
v(z) = \frac{1}{2\pi^2}\int_{\C}e_{-\lambda}(z)\left(ih(\lambda)\overline{\mu(z,\lambda)}-\frac{h(\lambda)}{\bar \lambda}\overline{\left(\frac{\partial \mu(z,\lambda)}{\partial z}\right)}\right)d\lambda\, d\bar \lambda.
\end{equation}
Formula \eqref{expl} for $v_1$ and $v_2$ yields
\begin{align*}
v_2(z)-v_1(z)&=\frac{1}{2\pi^2}\int_{\C}e_{-\lambda}(z)\Bigg[i(h_2(\lambda)-h_1(\lambda))\overline{\mu_2(z,\lambda)}\\
&\quad +ih_1(\lambda)(\overline{\mu_2(z,\lambda)}-\overline{\mu_1(z,\lambda)})\\
&\quad-\frac{h_2(\lambda)-h_1(\lambda)}{\bar \lambda}\overline{\left(\frac{\partial \mu_2(z,\lambda)}{\partial z}\right)}\\
&\quad-\frac{h_1(\lambda)}{\bar \lambda}\overline{\left(\frac{\partial \mu_2(z,\lambda)}{\partial z} -\frac{\partial \mu_1(z,\lambda)}{\partial z}\right)}\Bigg]d\lambda\, d \bar \lambda.
\end{align*}
Then, using several times H\"older's inequality, we find
\begin{align*}
|v_2(z)-v_1(z)|&\leq \frac{1}{2\pi^2}\Bigg(\|\mu_2(z,\cdot)\|_{L^{\infty}}\|h_2-h_1\|_{L^1}\\
&\quad+ \|h_1\|_{L^{\tilde p'}}\|\mu_2(z,\cdot)-\mu_1(z,\cdot)\|_{L^{\tilde p}}\\
&\quad+ \left\|\frac{h_2(\lambda)-h_1(\lambda)}{\bar \lambda}\right\|_{L^p}\left\|\frac{\partial \mu_2(z,\cdot)}{\partial z}\right\|_{L^{p'}}\\
&\quad+\left\|\frac{h_1(\lambda)}{\bar \lambda}\right\|_{L^{\tilde p'}}\left\|\frac{\partial \mu_2(z,\cdot)}{\partial z}-\frac{\partial \mu_1(z,\cdot)}{\partial z}\right\|_{L^{\tilde p}}\Bigg),
\end{align*}
for $1<p<2$, $\tilde p$ defined as in \eqref{deftilde} and $1/p + 1/p' = 1/{\tilde p} + 1/{\tilde p'} = 1$. From \eqref{normmu}, \eqref{prophlp}, the continuity of $h_j$ and Lemma \ref{lemh}, \cite[Lemma 2.1]{N4} (see the end of the proof of Lemma \ref{lemestmu} for more details), Lemma \ref{lemestmu}, Propositions \ref{prophdir2} and \ref{prophdir} we finally obtain
\begin{align*}
\|v_2-v_1\|_{L^{\infty}(D)} &\leq c(D,N,m,p)\bigg( \log(3 + \|\Phi_2 - \Phi_1\|_{H^{1/2}\to H^{-1/2}}^{-1})^{-(m-2)}\\
&\quad + \log(3 + \|\Phi_2 - \Phi_1\|_{H^{1/2}\to H^{-1/2}}^{-1})^{-(m+1-2/p)}\\
&\quad + \log(3 + \|\Phi_2 - \Phi_1\|_{H^{1/2}\to H^{-1/2}}^{-1})^{-(m-2/p)}\bigg)\\
&\leq c(D,N,m,p)\log(3 + \|\Phi_2 - \Phi_1\|_{H^{1/2}\to H^{-1/2}}^{-1})^{-(m-2)}.
\end{align*}
This finishes the proof of Theorem \ref{maintheo}.
\end{proof}

\begin{proof}[Proof of Theorem \ref{coro}] We first extend $\sigma$ on the whole plane by putting $\sigma (x) =1$ for $x \in \R^2 \setminus D$ (this extension is smooth by our hypothesis on $\sigma$). Now since $\sigma_j|_{\partial D}=1$ and $\frac{\partial \sigma_j}{\partial \nu}|_{\partial D} = 0$ for $j=1,2$, from \eqref{opchange} we deduce that
\begin{equation} \label{direq}
\Phi_j = \Lambda_j, \qquad j=1,2.
\end{equation}
In addition, from \eqref{estmunear} we get
\begin{equation}
\lim_{\lambda \to 0} \mu_j(z,\lambda) = \sigma_j^{1/2}(z), \qquad j=1,2;
\end{equation}
thus we obtain, using the fact that $\sigma_j$ is bounded from above and below, for $j=1,2$,
\begin{align} \label{ciao}
\|\sigma_2 - \sigma_1\|_{L^{\infty}(D)} &\leq c(N)\|\sigma^{1/2}_2 - \sigma^{1/2}_1\|_{L^{\infty}(D)}\\ \nonumber
&= c(N)\|\mu_2(\cdot,0)-\mu_1(\cdot,0)\|_{L^{\infty}(D)}.
\end{align}

Now fix $\alpha < m$ and take $p$ such that $$\max\left(1,\frac{2}{m-\alpha + 1}\right)<p < 2.$$
From Lemma \ref{prophv} we have
\begin{equation}
\|\mu_2(\cdot,0)-\mu_1(\cdot,0)\|_{L^{\infty}(D)} \leq c(D,N,p)\left\|\frac{h_2(\lambda) - h_1(\lambda)}{\bar \lambda}\right\|_{L^p(\C)\cap L^{p'}(\C)},
\end{equation}
where $1/p+1/p'=1$. From Proposition \ref{prophdir}
\begin{align*}
\left\|\frac{h_2(\lambda) - h_1(\lambda)}{\bar \lambda}\right\|_{L^p(\C)\cap L^{p'}(\C)} \! \! \! \! \! \! \! \! \! \! \! \! &\leq c(D,N,p)\log(3 + \|\Phi_2 - \Phi_1\|_{H^{1/2}\to H^{-1/2}}^{-1})^{-(m+1-2/p)}\\
&\leq c(D,N,p)\log(3 + \|\Phi_2 - \Phi_1\|_{H^{1/2}\to H^{-1/2}}^{-1})^{-\alpha}\\
&= c(D,N,p)\log(3 + \|\Lambda_2 - \Lambda_1\|_{H^{1/2}\to H^{-1/2}}^{-1})^{-\alpha},
\end{align*}
from \eqref{direq} and since $\alpha < m+1 - \frac{2}{p}$. Theorem \ref{coro} is thus proved.
\end{proof}


\begin{thebibliography}{99}
\bibitem{A} G. Alessandrini, \textit{Stable determination of conductivity by boundary measurements}, Appl. Anal. \textbf{27}, 1988, no. 1, 153--172.
\bibitem{AP} K. Astala, L. P\"aiv\"arinta, \textit{Calder\'on's inverse conductivity problem in the plane}, Ann. Math. \textbf{163}, 2006, 265--299.
\bibitem{BBR} J. A. Barceló, T. Barceló, A. Ruiz, \textit{Stability of the inverse conductivity problem in the plane for less regular conductivities}, J. Diff. Equations \textbf{173}, 2001, 231--270.
\bibitem{BFR} T. Barceló, D. Faraco, A. Ruiz, \textit{Stability of Calderón inverse conductivity problem in the plane}, J Math Pures Appl. \textbf{88}, 2007, no. 6, 522--556.
\bibitem{BC} R. Beals, R. R. Coifman, \textit{Multidimensional inverse scatterings and nonlinear partial differential equations}, Pseudodifferential operators and applications (Notre Dame, Ind., 1984), 45--70, Proc. Sympos. Pure Math., \textbf{43}, Amer. Math. Soc., Providence, RI, 1985.
\bibitem{B} A. L. Bukhgeim, \textit{Recovering a potential from Cauchy data in the two-dimensional case}, J. Inverse Ill-Posed Probl. \textbf{16}, 2008,  no. 1, 19--33.
\bibitem{C} A. P. Calder\'on, \textit{On an inverse boundary problem}, Seminar on Numerical Analysis and its Applications to Continuum Physics, Soc. Brasiliera de Matematica, Rio de Janeiro, 1980, 61--73.
\bibitem{DKN} B. A. Dubrovin, I. M. Krichever, S. P. Novikov, \textit{The Schr\"odinger equation in a periodic field and Riemann surfaces}, Dokl. Akad. Nauk SSSR \textbf{229}, 1976, no. 1, 15--18.
\bibitem{F} L. D. Faddeev, \textit{Growing solutions of the Schr\"odinger equation}, Dokl. Akad. Nauk SSSR \textbf{165}, 1965, no. 3, 514--517.
\bibitem{G} I. M. Gel'fand, \textit{Some aspects of functional analysis and algebra}, Proceedings of the International Congress of Mathematicians, Amsterdam, 1954, \textbf{1}, 253--276. Erven P. Noordhoff N.V., Groningen; North-Holland Publishing Co., Amsterdam.
\bibitem{GN} P. G. Grinevich, S. P. Novikov, \textit{Two-dimensional “inverse scattering problem” for negative energies and generalized-analytic functions. I. Energies below the ground state}, Funct. Anal. and Appl. \textbf{22}, 1988, no. 1, 19--27.
\bibitem{HN} G. M. Henkin, R. G. Novikov, \textit{The $ \bar\partial$-equation in the multidimensional inverse scattering problem}, Russian Mathematical Surveys \textbf{42}, 1987, no. 3, 109--180.
\bibitem{I} M. Isaev, \textit{Exponential instability in the Gel'fand inverse problem on the energy intervals}, J. Inverse Ill-Posed Probl. \textbf{19}, 2011,  no. 3,  453--472; e-print arXiv:1012.2193.
\bibitem{L} L. Liu, \textit{Stability Estimates for the Two-Dimensional Inverse Conductivity Problem},
Ph.D. thesis, Department of Mathematics, University of Rochester, New York, 1997.
\bibitem{M} N. Mandache, \textit{Exponential instability in an inverse problem of the Schr\"odinger equation}, Inverse Problems \textbf{17}, 2001, no. 5, 1435--1444.
\bibitem{Na} A. Nachman, \textit{Global uniqueness for a two-dimensional inverse boundary value problem}, Ann. Math. \textbf{143}, 1996, 71--96.
\bibitem{N1} R. G. Novikov, \textit{Multidimensional inverse spectral problem for the equation $-\Delta \psi + (v(x) - Eu(x))\psi = 0$}, Funkt. Anal. i Pril. \textbf{22}, 1988, no. 4, 11--22 (in Russian); English
Transl.: Funct. Anal. and Appl. \textbf{22}, 1988, no. 4, 263--272.
\bibitem{N5} R. G. Novikov, \textit{The inverse scattering problem on a fixed energy level for the two-dimensional Schrödinger operator}, J. Funct. Anal. \textbf{103}, 1992, no. 2, 409--463.
\bibitem{N4} R. G., Novikov, \textit{Approximate solution of the inverse problem of quantum scattering theory with fixed energy in dimension $2$}, (Russian)  Tr. Mat. Inst. Steklova  \textbf{225}, 1999,  Solitony Geom. Topol. na Perekrest., 301--318;  translation in  Proc. Steklov Inst. Math.  \textbf{225}, 1999,  no. 2, 285--302.
\bibitem{N3} R. G. Novikov, \textit{Formulae and equations for finding scattering data from the Dirichlet-to-Neumann map with nonzero background potential}, Inv. Problems 21, 2005, no. 1, 257--270.
\bibitem{N2} R. G. Novikov, \textit{New global stability estimates for the Gel'fand-Calderon inverse problem}, Inv. Problems \textbf{27}, 2011, no. 1, 015001.
\bibitem{NN} R. G. Novikov, N. N. Novikova, \textit{On stable determination of potential by boundary measurements}, ESAIM: Proc. \textbf{26}, 2009, 94--99.
\bibitem{NS} R. G. Novikov, M. Santacesaria, \textit{A global stability estimate for the Gel'fand-Calder\'on inverse problem in two dimensions}, J. Inverse Ill-Posed Probl. \textbf{18}, 2010, no. 7, 765--785.
\bibitem{V} I. N. Vekua, \textit{Generalized Analytic Functions}, Pergamon Press Ltd. 1962.
\end{thebibliography}
\end{document}